\documentclass[12pt]{amsart}
\usepackage{amsmath,amssymb,amsfonts,enumerate,amsthm}

\numberwithin{equation}{section}

\newtheorem{thm}{Theorem}[section]
\newtheorem{lem}[thm]{Lemma}
\newtheorem{cor}[thm]{Corollary}
\newtheorem{prop}[thm]{Proposition}
\newtheorem{rem}[thm]{Remark}
\newtheorem{defin}[thm]{Definition}
\newtheorem{defins}[thm]{Definitions}

\newtheorem{nota rem}[thm]{Notation and Remark}
\newtheorem{defen}[thm]{Definition}
\newtheorem{defen rem}[thm]{Definition and Remark}
\def\e{ \begin{enumerate} \it }
\def\ee{\end{enumerate} }
\def\en{\operatorname{end}}
\def\ker{\operatorname{ker}}
\def\Var{\operatorname{Var}}
\def\coker{\operatorname{coker}}
\def\Tor{\operatorname{Tor}}
\def\Att{\operatorname{Att}}
\def\Var{\operatorname{Var}}
\def\pd{\operatorname{pd}}
\def\hei{\operatorname{ht}}

\def\pd{\operatorname{pd}}

\def\Ext{\operatorname{Ext}}
\def\Hom{\operatorname{Hom}}
\def\Spec{\operatorname{Spec}}

\def\Var{\operatorname{Var}}
\def\beg{\operatorname{beg}}
\def\grade{\operatorname{grade}}

\def\cd{\operatorname{cd}}
\def\grade{\operatorname{grade}}
\def\im{\operatorname{im}}

\def\beg{\operatorname{beg}}
\def\en{\operatorname{end}}
\def\Supp{\operatorname{Supp}}

\def\length{\operatorname{length}}
\def\Ass{\operatorname{Ass}}
\def\dim{\operatorname{dim}}
\newcommand{\p}{\mathfrak{p}}
\newcommand{\fb}{\mathfrak{b}}
\newcommand{\fa}{\mathfrak{a}}

\newcommand{\m}{\mathfrak{m}}

\newcommand{\Z}{\mathbb{Z}}

\newcommand{\N}{\mathbb{N}}

\begin{document}
\bibliographystyle{amsplain}

\author{M. Jahangiri}
\address{School of Mathematics, Institute for Research in Fundamental Sciences (IPM), P. O. Box 19395-5749, Tehran, Iran -and- School of Mathematics and Computer Sciences, Damghan University, Damghan, Iran.}
\email{mjahangiri@ipm.ir\\
jahangiri.maryam@gmail.com}

\author{N. Shirmohammadi}
\address{School of Mathematics, Institute for Research in Fundamental Sciences (IPM), P. O. Box 19395-5749, Tehran, Iran -and- Department of Mathematics, University of Tabriz, Tabriz , Iran.}
\email{shirmohammadi@tabrizu.ac.ir }

\author{Sh. Tahamtan}
\address{Islamic Azad University, Borujerd branch, Boroujerd, Iran.}
\email{taham\_sh@yahoo.com }

\thanks{\today }
\subjclass[2000]{13A02, 13E10, 13D45. The first author was in part supported by a grant from IPM (No. 89130115) The second author was in part supported by a grant from IPM (No. 89130057)}

\title[Tameness and Artinianness of graded generalized local cohomology modules]
{Tameness and Artinianness of graded generalized local cohomology modules }

%---------------------------------------------------------------------------------------------------------------------------abstract
\begin{abstract}
Let $R=\bigoplus_{n\geq 0}R_n$, $\fa\supseteq  \bigoplus_{n> 0}R_n$ and $M$ and $N$ be a standard graded ring, an ideal of $R$ and two finitely generated graded $R$-modules, respectively. This paper studies the homogeneous components of graded generalized local cohomology modules. First of all, we show that for all $i\geq 0$, $H^i_{\fa}(M, N)_n$, the $n$-th graded component of the $i$-th generalized local cohomology module of $M$ and $N$ with respect to $\fa$, vanishes for all $n\gg 0$. Furthermore, some sufficient conditions are proposed to satisfy the equality  $\sup\{\en(H^i_{\fa}(M, N))| i\geq 0\}= \sup\{\en(H^i_{R_+}(M, N))| i\geq 0\}$.

Some sufficient conditions are also proposed for tameness of $H^i_{\fa}(M, N)$ such that $i= f_{\fa}^{R_+}(M, N)$ or $i= \cd_{\fa}(M, N)$, where $f_{\fa}^{R_+}(M, N)$ and $\cd_{\fa}(M, N)$ denote the $R_+$-finiteness dimension and the cohomological dimension of $M$ and $N$ with respect to $\fa$, respectively. We finally consider the Artinian property of some submodules and quotient modules of $H^j_{\fa}(M, N)$, where $j$ is the first or last non-minimax level of $H^i_{\fa}(M, N)$.
\end{abstract}
\maketitle
%-----------------------------------------------------------------------------------------------------------------------introduction
\section{introduction}
Assume that $R$ is a commutative Noetherian ring with identity and all modules are unitary. Let $\fa$, $\mathfrak{\zeta}(R)$ and $\N_0$ ($\N$) be an ideal of $R$, the category of all $R$-modules and $R$-homomorphisms and the set of non-negative (positive) integers, respectively.

For $i\in\N_0$, the $i$-th generalized local cohomology functor with respect to $\fa$ is a generalization of the $i$-th local cohomology functor with respect to $\fa$, i.e. $H^i_{\fa}(-)= \underset{ \overset {\longrightarrow}{n\in \mathbb{N}}}{\mathrm
 {lim}
 }\Ext_R^i(R/\fa^n, -)$ (\cite{bsh}). It is defined, by Herzog (\cite{h}), as follows:

 \[
 H^{i}_{\fa} (-, -): \mathfrak{\zeta}(R)\times \mathfrak{\zeta}(R)\rightarrow \mathfrak{\zeta}(R)\]
  \[H^{i}_{\fa} (M, N)=\underset{ \overset {\longrightarrow}{n\in \mathbb{N}}}{\mathrm
 {lim}
 }\,\,  \mathrm {Ext}^{i}_{R}(M/ \fa^{n}M, N).\]
 For all $R$-modules $M$ and $N$, $H^{i}_{\fa} (M, N)$ is called the $i$-th generalized local cohomology module of $M$ and $N$ with respect to $\fa$. These functors coincide when $M= R$ and have been studied by many authors (see for instance  \cite{kh}, \cite{s}-\cite{z}). One of the most interesting problems concerning these modules is their vanishing problem. Although, $H^i_{\fa}(N)= 0$ for sufficiently large values of $i$ (\cite[3.3.1]{bsh}), $H^i_{\fa}(M, N)$ can be non-zero for all $i\in \N$.
 %(see \cite{D}).
However, it can be proved that $H^i_{\fa}(M, N)= 0$ for all $i\gg 0$ in some cases. Such as when $M$ and $N$ are finitely generated and $\pd_R(M)< \infty$, where $\pd_R(M)$ denote the projective dimension of $M$ (\cite{y}).

Now assume  $R=\bigoplus_{n\in \N_0}R_n$ is a standard graded ring, $\fa$ is a homogeneous ideal of $R$ and $M$ and $N$ are graded $R$-modules. It is wellknown that $H^i_{\fa}(M, N)$ caries a natural grading. Furthermore, assume that $R_+=\bigoplus_{n\in \N}R_n$ denote the irrelevant ideal of $R$, $M$ and $N$ are finitely generated and $\Z$ denotes the set of integers. Then it is shown in \cite{kh} that the $n$-th graded component of $H^i_{R_+}(M, N)$, i.e. $H^i_{R_+}(M, N)_n$, is a finitely generated $R_0$-module for all $n\in \Z$ and it vanishes for all $n\gg 0$.
 Also, in \cite{z}, it is proved that when $(R_0, \m_0)$ is local and $ \pd_R(M)< \infty$ then $H^i_{R_+}(M, N)= 0$ for all $i> \pd_R(M)+ \dim(N/ \m_0N)=: c$ and that $H^c_{R_+}(M, N)$ is tame. It is a generalization of the obtained result in \cite[4.8(e)]{b}. The tameness of $H^i_{\fa}(M, N)$ is related to the case that $H^i_{\fa}(M, N)_n =0$ for all $n\ll 0$ or
 $H^i_{\fa}(M, N)_n \neq 0$ for all $n\ll 0$.

 Actually, there are lots of interest in the study of tameness property of these modules. Although, Chardin et al. in \cite{cchs} showed the tameness of these modules is not true in a general case, it holds in some cases (see \cite{b}). Therefore, finding other sufficient conditions for tameness of these modules is motivated.
 In this paper, some sufficient conditions are proposed for the tameness of these modules. The stability of the set of their associated prime ideals
 $\Ass_{R_0}(H^i_{\fa}(M, N)_n)$, when $n\rightarrow -\infty$, is then studies.
 By a modification of Singh's example (\cite{bks}), these sets might be non stable when $n\rightarrow -\infty$, but in some cases it holds. The cases consist in $\fa =R_+$ and $(R_0, \m_0)$ is local of dimension $\leq 1$ or $i\leq f_{R_+}(M, N)$, where $f_{R_+}(M, N)$ denotes the first integer $j$ such that $H^j_{R_+}(M, N)$ is not finitely generated (see \cite{kh}).

 Let $\fa_0$ and $\fa:= \fa_0+ R_+$, respectively, be an ideal of $R_0$ and an ideal of $R$ which contains the irrelevant ideal. Also, let $M$ and $N$ be two finitely generated graded $R$-modules. The structure of this paper is organized as follows:

 In Section 2, we first study some general properties of the components of $H^i_{\fa}(M, N)$. In particular, we show that $H^i_{\fa}(M, N)_n= 0$ for all $n\gg 0$ and they are finitely generated in some cases (\ref{gen}). We will then show a nice property of these modules. The property states that when $(R_0, \m_0)$ is local and $\pd_R(M)< \infty$, then $$\sup\{\en(H^i_{\fa}(M, N))| i\in \N_0\}= \sup\{\en(H^i_{R_+}(M, N))| i\in \N_0\},$$
 in the case where $\fa_0$ is principal or $R$ is Cohen-Macaulay (\ref{cm} and \ref{1}).

 In Section 3, we give an upper bound for
 $$\cd_{\fa}(M, N):= \sup \{i\in \N_0| H^i_{\fa}(M, N)\neq 0\},$$
 when $\pd_R(M)< \infty$. More precisely, we show that
 in this case
 $$\cd_{\fa}(M, N)\leq \pd_R(M)+ \max\{\cd_{\fa}(\Ext^{i}_{R}(M, N))| i\in \N_0\}:= c,$$
 where for any $R$-module $X$, $\cd_{\fa}(X):= \sup \{i\in \N_0| H^i_{\fa}(X)\neq 0\}$ (\ref{cd}), and that $H^c_{\fa}(M, N)$ is tame in some cases (\ref{tame}).

We will also show that for each $i\leq f_{\fa}^{R_+}(M, N)$ there exists a finite subset $X$ of $\Spec(R_0)$ such that $\Ass_{R_0}(H^i_{\fa}(M, N)_n)= X$ for all $n\ll 0$, where
 $$f_{\fa}^{R_+}(M, N):= \inf \{i\in \N_0| R_+\nsubseteq \sqrt{0:_R H^i_{\fa}(M, N)}\},$$
 is the $R_+$-finiteness dimension of $M$ and $N$ with respect to $\fa$ (\ref{ass}). It implies that $H^i_{\fa}(M, N)$ is tame for all $i\leq f_{\fa}^{R_+}(M, N)$.

In the last Section, we study the tameness and Artinianness of
 $H^i_{\fa}(M, N)$ in the first and last non-minimax levels.

Throughout the paper, $R= \bigoplus_{n\in \N_0}R_n$ is a standard graded Noetherian ring, $R_+= \bigoplus_{n\in \N}R_n$ is the irrelevant ideal of $R$, $\fa_0$ is an ideal of $R_0$ and $\fa:= \fa_0+ R_+$. Also, $M$ and $N$ are two finitely generated graded $R$-modules.

 %%%%%%%%%%%%%%%%%%%%%%%%%%%%%%%%%%%%%%%%%%%%%%%%%%%%%%%%%%%%%%%%%%%%%%%%%%%%%%%%%%%%%%%%%%%%%%%%%%%%%%%%%%5
\section{On the behavior of $H^i_{\fa}(M, N)_n$ for $n\gg 0$}

It is wellknown that $H^i_{\fa}(M, N)$ is a graded $R$-module and $H^i_{R_+}(M, N)_n$ is a finitely generated $R_0$-module for all $n\in \Z$ and it vanishes for all $n\gg 0$ (see \cite[3.2]{kh}).

In this section we are going to study some general properties of graded components of generalized local cohomology module $H^i_{\fa}(M, N)$ with respect to an ideal $\fa$ containing the irrelevant ideal.

\begin{rem}\label{seq}%------------------------------------------------------------------------------------------------------------------------------

(\textrm{i}) Let $L$ and $K$ be two graded $R$-modules, $\fb$ be a homogenous ideal of $R$ and $x$ a homogenous element of $R$. Then, in view of \cite[3.1]{dh},  there exists a long exact sequence
  $$\cdots \longrightarrow H^i_{\fb+xR}(L, K)\longrightarrow H^i_{\fb}(L, K)\longrightarrow H^i_{\fb}(L, K)_x\longrightarrow H^{i+1}_{\fb+xR}(L, K)\longrightarrow \cdots $$
  of graded $R$-modules.

(\textrm{ii})
A sequence
$x_{1},...,x_{t}$ of homogeneous elements of $R_+$ is said to be a
homogeneous $\fa$-filter regular sequence on $M$  if for all $ i=
1,...,t$, $x_{i} \notin \mathfrak{p}$ for all $\mathfrak{p}\in
\Ass_{R}(M/ (x_{1},...,x_{i-1})M)\setminus V(\fa)$,
where
$V(\fa)$ is the set of prime ideals of $R$ containing $\fa$. It is straightforward to see that if $\Supp_{ R}(
 M/ R_{+}M)\nsubseteq V(\mathfrak{a})$, then
all maximal homogeneous $\fa$-filter regular sequences on
$M$ in $R_+$ have the same length. We denote, in this case, the common
length of all maximal homogeneous $\fa$-filter regular
sequences on $M$ in $R_+$ by $f-\grade (\mathfrak{a}, R_{+ }, M )$.
Also, we set $f-\grade (\mathfrak{a}, R_{+},
 M) = \infty$ whenever $\Supp_{ R}(
M/ R_{+}M)\subseteq V (\mathfrak{a} )$.

(\textrm{iii}) Let $X$ and $Y$ be two $R$-modules and $E_Y^{\bullet}$ be an injective resolution of $Y$. Then, in view of \cite{s}, one has
$$H^i_{\fa}(X, Y)\cong H^i(\Gamma_{\fa}(\Hom_R(X, E_Y^{\bullet})))\cong H^i(\Hom_R(X, \Gamma_{\fa}(E_Y^{\bullet}))).$$
Therefore, if $Y$ is an $\fa$-torsion $R$-module then, using \cite[2.1.6]{bsh}, $H^i_{\fa}(X, Y)\cong \Ext^i_R(X, Y)$.

(\textrm{iv}) If $0\to X\to Y\to Z\to 0$ is an exact sequence of graded $R$-modules and $R$-homomorphisms, then for any graded $R$-module $K$ there are long exact sequences of graded generalized local cohomology modules
$$\cdots \to H^i_{\fa}(K, X)\to H^i_{\fa}(K, Y)\to H^i_{\fa}(K, Z)\to H^{i+1}_{\fa}(K, X)\to \cdots$$
and
$$\cdots \to H^i_{\fa}(Z, K)\to H^i_{\fa}(Y, K)\to H^i_{\fa}(X, K)\to H^{i+1}_{\fa}(Z, K)\to \cdots.$$
\end{rem}

%---------------------------------------------------------------------------------------------------------------------------------
\begin{defen rem}\label{g}

 Let $(R_0, \m_0)$ be local and $\fb$ be a homogenous ideal of $R$. Define
  $$g_{\fb}(M, N):= \inf \{i\in \N_0| \forall j< i, \ \ \length_{R_0}(H^j_{\fb}(M, N)_n)< \infty \ \ \ \forall n\ll 0\},$$
  as the \textit{cohomological finite length dimension of $M$ and $N$ with respect to $\fa$}.

  If $\fa_0\subseteq \fb_0\subseteq \m_0$ be two ideals of $R_0$, then using \ref{seq}(i),
 it is straightforward to see that $g_{\fa_0+ R_+}(M, N)\leq g_{\fb_0+ R_+}(M, N)$.
\end{defen rem}

In the following theorem we study vanishing and Noetherian property of graded components of $H^i_{\fa}(M, N)$. To this end, we use the concept of the \textit{end} and \textit{beginning} (beg(X)) of a graded $R$-module $X= \oplus_{n\in \Z}X_n$, which are defined by
\[\en(X):= \sup\{n\in \Z| X_n\neq 0\} \ \ \textrm{and} \ \ \beg(X):= \inf\{n\in \Z| X_n\neq 0\}.\]
(Note that $\en(X)$ could be $\infty$, and that the supremum of the empty set of integers is to be taken as $-\infty$; similar comments apply to $\beg(X)$.)

\begin{thm}\label{gen}%---------------------------------------------------------------------------------------------------------------------------------
The following statements hold.
\\
(\textrm{i}) For all $i\in \N_0$, $H^i_{\fa}(M, N)_n =0$ for all $n\geq \sup\{\en(H^j_{\fa}(N))| j\in \N_0\}- \beg(M)$;
\\
(\textrm{ii})  let $\pd_R(M)< \infty$. Then $H^i_{\fa}(M, N)_n $ is a finitely generated graded $R_0$-module for all $n\in \Z$ and all $i\leq f-\grade(\fa, R_+, N)$;
\\
(\textrm{iii}) let $(R_0, \m_0)$ be local. Then $H^i_{\fa}(M, N)_n $ is a finitely generated $R_0$-module for all $n\ll 0$ and all $i\leq g_{R_+}(M, N)$.

\end{thm}

\begin{proof}

 (\textrm{i}) Let $i\in \N_0$. Then, using \cite[2.4]{sh}, we have

 \begin{align}
 \en(H^i_{\fa}(M, N))&= \en(H^i(\Hom_R(M, \Gamma_{\fa}(^*E_R(N))))) \nonumber \\
 &\leq \sup\{\en(\Hom_R(M, \Gamma_{\fa}(^*E_R^j(N))))| j\in \N_0\} \nonumber \\
 &\leq \sup\{\en(\Gamma_{\fa}(^*E_R^j(N)))| j\in \N_0\}- \beg(M) \nonumber \\
 &= \sup\{t\in \Z| \exists \p\in V(\fa), ^*E_R(R/\p)(-t)\leq ^*E_R^j(N) \ \mbox{ for some} \ j\in \N_0\}- \beg(M) \nonumber \\
 &= \sup\{\en(H^j_{\fa}(N))| j\in \N_0\}- \beg(M), \nonumber
 \end{align}

 where, for each $j\in \N_0$, $^*E_R^j(N)$ denote the $j$-th term in a minimal graded injective resolution $^*E_R(N)$ of $N$.

(ii) One can prove the claim using induction on $\pd_R(M)$ and \ref{seq}(iv) in conjunction with \cite[1.7]{jz}.

 (iii) Assume that $\fa_0=(x_1,\cdots, x_n)\subseteq \m_0$. We use induction on $n$. Let $n=1$ and $i\leq g_{R_+}(M, N)$. Since $H^{i- 1}_{R_+}(M, N)_n$ is of finite length and $x_1\in \m_0$, so $(H^{i-1}_{R_+}(M, N)_n)_{x_1}=0$ for all $n\ll 0$. Now, in this case the result follows using the exact sequence
  $$ (H^{i-1}_{R_+}(M, N)_n)_x\longrightarrow H^i_{\fa}(M, N)_n\longrightarrow H^i_{R_+}(M, N)_n$$
  and \cite[3.2]{kh}. For the case $n> 0$, one can use the same argument as used in the case $n= 1$ in conjunction with \ref{g}.
\end{proof}

The above theorem shows that $H^i_{\fa}(M, N)_n = 0$ for sufficiently large values of $n$. But, according to \cite[1.3]{jz}, these modules can be non-Noetherian in general.

In \cite[3.2]{sh}, \cite[2.3]{Hy} and \cite[1.6]{jz} it is shown that in the case where $(R_0, \m_0)$ is a local ring and $\fa_0\subseteq \m_0$, we have $$\sup \{\en(H^i_{R_+}(N))| i\geq 0 \}= \sup \{\en(H^i_{\fa_0+ R_+}(N))| i\geq 0 \}.$$

 Now, it is natural to ask if
$$\sup \{\en(H^i_{R_+}(M, N))| i\geq 0 \}= \sup \{\en(H^i_{\fa}(M, N))| i\geq 0 \}.$$

(Note that if $\pd_R(M)< \infty$ then, in view of \ref{gen}(i) and \cite[2.5]{y}, $$\sup \{\en(H^i_{\fa_0+ R_+}(M, N))| i\geq 0 \}<\infty.)$$
In the rest of this section we are going to answer to this question in some special cases. To do this, it will be convenient to have available a notation. Define $$a^*_{\fa}(M, N):= \sup \{\en(H^i_{\fa}(M, N))| i\geq 0 \}.$$In the rest of this section, we assume that $(R_0, \m_0)$ is a local ring and $\fa_0\subseteq \m_0$. We use $\m:= \m_0+ R_+$ to denote the unique graded maximal ideal of $R$.

\begin{lem}\label{1}%-----------------------------------------------------------------------------------------------------------------------------
Let $\fa_0=xR_0\subseteq \m_0$ be a principal ideal. Then $a^*_{\fa}(M, N)= a^*_{R_+}(M, N)$.
\end{lem}
\begin{proof}
In view of \ref{seq}(i), for all $n\in \Z$ there exists a long exact sequence
$$\cdots\to (H^{i-1}_{R_+}(M, N)_n)_x\to H^i_{\fa}(M, N)_n\to H^i_{R_+}(M, N)_n\to (H^{i}_{R_+}(M, N)_n)_x\to H^{i+ 1}_{\fa}(M, N)_n \to \cdots$$
of generalized local cohomology modules. As $(R_0, \m_0)$ is local and $H^i_{R_+}(M, N)_n$ is a finitely generated $R_0$-module for all $n\in \Z$, the above exact sequence implies that for all $i\in \N_0$, $H^i_{R_+}(M, N)_n= 0$ if and only if $H^i_{\fa}(M, N)_n= 0$ for all $i\in \N_0$, where $n\in \Z$. This proves the claim.
\end{proof}

\begin{thm}%-----------------------------------------------------------------------------------------------------------------------------
Assume that for all $i\in \N_0$, all ideal $\fb_0\subseteq \fa_0$ of $R_0$ and all $n\in \Z$, $\bigcap_{t\in \N}\fa_0^t H^i_{\fb_0+ R_+}(M, N)_n= 0$. Then, $a^*_{\fa}(M, N)= a^*_{R_+}(M, N)$.
\end{thm}
\begin{proof}
Let $\fa_0= (x_1,\cdots, x_n)R_0$. Then the result follows by similar argument as used in \ref{1} in conjunction with induction on $n$.
\end{proof}

\begin{cor}%-----------------------------------------------------------------------------------------------------------------------------
Let $H^i_{\fb_0+ R_+}(M, N)_n$ be finitely generated for all $i\in \N_0$, $n\in \Z$ and ideal $\fb_0\subseteq \fa_0$ of $R_0$. Then $a^*_{\fa}(M, N)= a^*_{R_+}(M, N)$.
\end{cor}

\begin{rem}\label{iff}%-----------------------------------------------------------------------------------------------------------------------------
Let $\fb_0$ be a second ideal of $R_0$ such that $\fa_0\subseteq \fb_0$. Then, using the exact sequence \ref{seq}(i) one can see that $a^*_{\fb_0+ R_+}(M, N)\leq a^*_{\fa_0+ R_+}(M, N)$. So, $a^*_{\m}(M, N)\leq a^*_{\fa_0+ R_+}(M, N)\leq a^*_{R_+}(M, N)$. Hence, $a^*_{\fa}(M, N)= a^*_{R_+}(M, N)$, for all proper homogeneous ideal $\fa\supseteq R_+$, if and only if $a^*_{\m}(M, N)\geq a^*_{R_+}(M, N)$.
\end{rem}

We now remind a duality theorem of generalized local cohomology modules. It is needed to prove the last theorem of this section.

\begin{thm}\label{dual}%-----------------------------------------------------------------------------------------------------------------------------
Let $R$ be Cohen-Macaulay with $\dim(R)= d$ which posses a canonical module $\omega_R$  and let $^*E_R(R/ \m)$ be the graded injective envelope of $R/ \m$. Also, assume that the projective dimension of $M$ or $N$ is finite. Then, there exist homogeneous isomorphisms $$H^{d- i}_{\m}(M, N)\cong ^*\Hom_R(\Ext^i_R(N, \omega_R\otimes_R M), ^*E_R(R/ \m)),$$
for all $i\in \N_0$.
\end{thm}
\begin{proof}
See \cite[3.8(ii)]{md}.
\end{proof}

\begin{thm}\label{cm}%-----------------------------------------------------------------------------------------------------------------------------
Let $R$ be Cohen-Macaulay with $\dim(R)= d$ and assume that the projective dimension of $M$ or $N$ is finite. Then, $a^*_{\fa}(M, N)= a^*_{R_+}(M, N)$.
\end{thm}
\begin{proof}
Using \ref{iff}, it is enough to show that $a^*_{\m}(M, N)\geq a^*_{R_+}(M, N)$.
Let $\widehat{R_0}$ denote the $\m_0$-adic completion of $R_0$. Then, in view of the flat base change property of generalized local cohomology modules (\cite[4(ii)]{kh}), for all $i\in \N_0$ we have an isomorphism of graded modules
$H^{i}_{\m}(M, N)\otimes_{R_0}\widehat{R_0}\cong H^{i}_{\m'}(M\otimes_{R_0}\widehat{R_0}, N\otimes_{R_0}\widehat{R_0})$, where $\m'$ is the image of $\m$ under the faithful flat homomorphism $R\longrightarrow R\otimes_{R_0}\widehat{R_0}$. Therefore, for all $i\in \N_0$ and all $n\in \Z$, $H^{i}_{\m}(M, N)_n= 0$ if and only if $H^{i}_{\m'}(M\otimes_{R_0}\widehat{R_0}, N\otimes_{R_0}\widehat{R_0})_n= 0$. So, replacing $R$ with $R\otimes_{R_0}\widehat{R_0}$, we may assume that $R$ is a homomorphic image of a Gornestein ring. Which implies that it admits a canonical module $\omega_R$.

Set $s:= a^*_{\m}(M, N)$. In view of \ref{dual} and \cite[13.4.5(i), (iv)]{bsh}, there exist following isomorphisms
\begin{align}
H^{i}_{\m}(M, N)_n&\cong ^*\Hom_R(\Ext^{d- i}_R(N, \omega_R\otimes_R M), ^*E_R(R/ \m))_n \nonumber\\
&\cong \Hom_{R_0}(\Ext^{d- i}_R(N, \omega_R\otimes_R M)_{-n}, E_0), \nonumber
\end{align}

where $E_0= E_{R_0}(R_0/ \m_0)$. So, we have
\[    \hspace{3cm} \Ext^{d- i}_R(N, \omega_R\otimes_R M)_{-n}= 0 \ \ \mbox{for all} \ \ i\in \N_0 \ \ \mbox{and all} \ n> s.\hspace{3cm}(1)
\]

Now, let $\p_0\in \Spec (R_0)$, again using \ref{dual} and the fact that $R_{\p_0}$ admits a canonical module and that $\omega_{R_{\p_0}}\cong (\omega_{R})_{\p_0}$(\cite[3.3.5]{bh}), we conclude

\begin{align}
 H^{i}_{\p_0 R_{\p_0}+ (R_{\p_0})_+}(M_{\p_0}, N_{\p_0})_n
&\cong \Hom_{R_{\p_0}}(\Ext^{\hei(\p_0+ R_+)- i}_{R_{\p_0}}(N_{\p_0}, \omega_{R_{\p_0}}\otimes_{R_{\p_0}} M_{\p_0}), E_{R_0}(R_0/ \p_0)_{\p_0})_n \nonumber \\
&\cong (\Hom_{R_0}(\Ext^{\hei(\p_0+ R_+)- i}_R(N, \omega_R\otimes_R M)_{-n}, E_{R_0}(R_0/ \p_0))_{\p_0}.\nonumber
\end{align}

So, in view of $(1)$,

\[
\hspace{3cm} H^{i}_{\p_0R_{\p_0}+ (R_{\p_0})_+}(M_{\p_0}, N_{\p_0})_n= 0 \ \mbox{ for all} \ i\in \N_0 \ \mbox{and all} \ n> s.\hspace{2.9cm}(2)
\]

Next, we use induction on $\dim(R_0)$ to prove that $H^{i}_{R_+}(M, N)_n= 0$ for all $n> a^*_{\m}(M, N)$ and all $i\in \N_0$. In the case where $\dim(R_0)= 0$, we have
$$H^{i}_{ R_+}(M, N)_{n}\cong H^{i}_{\m_0 + R_+}(M, N)_n= 0 \ \ \ \mbox{for all } \ i\in \N_0 \ \ \mbox{and all}  \ \ n> a^*_{\m}(M, N).$$
Now, let $\dim(R_0)> 0$ and $\p_0\in \Spec (R_0)\backslash \{\m_0\}$. Then, using $(2)$ and inductive hypothesis, it is concluded that
$(H^{i}_{R_+}(M, N)_n)_{\p_0}\cong H^{i}_{(R_{\p_0})_+}(M_{\p_0}, N_{\p_0})_n= 0$ for all $i\in \N_0$ and all $n> s$. Therefore $\Supp_{R_0}(H^{i}_{R_+}(M, N)_n)\subseteq \{\m_0\}$, which implies that $H^{i}_{R_+}(M, N)_n$ has finite length for all $i\in \N_0$ and all $n> s$. Now, the convergence of spectral sequences
$$(E_2^{i, j})_n=H^i_{\m_0}(H^{j}_{R_+}(M, N)_n) \underset{i}{\Rightarrow}H^{i+ j}_{\m}(M, N)_n$$
(\cite[11.38]{r}),
in conjunction with the fact that $H^i_{\m_0}(H^{j}_{R_+}(M, N)_n)= 0$ for all $n> s$, all $j\in \N_0$ and all $i> 0$,
shows that $H^{i}_{R_+}(M, N)_n\cong H^{i}_{\m}(M, N)_n= 0$ for all $i\in \N_0$ and all $n> s$. Hence, $a^*_{\m}(M, N)\geq a^*_{R_+}(M, N)$, as desired.
\end{proof}

%%%%%%%%%%%%%%%%%%%%%%%%%%%%%%%%%%%%%%%%%%%%%%%%%%%%%%%%%%%%%%%%%%%%%%%%%%%%%%%%%%%%%%%%%%%%%%%%%%%%%%%%%%%%%%%%%%%%%%%%%%%%%%%%%%
\section{Tameness at finiteness dimension and almost top levels}

As we have seen in \ref{gen}, $H^i_{\fa}(M, N)_n =0$ for all $n\gg 0$ and all $i\in \N_0$. In this section we are going to study the asymptotic behavior of $H^i_{\fa}(M, N)_n $ when $n\rightarrow -\infty$. In particular, we will show that $H^i_{\fa}(M, N)$ is \emph{Tame} (in the sense that $H^i_{\fa}(M, N)_n= 0$ for all $n\ll 0$ or $H^i_{\fa}(M, N)_n\neq 0$ for all $n\ll 0$) in the first and last "\emph{nontrivial case}".

In \cite[2.3]{b}, it is shown that whenever $(R_0, \m_0)$ is local and $\cd_{R_+}(N)> 0$, $H^{\cd_{R_+}(N)}_{R_+}(N)_n\neq 0 $ for all $n\ll 0$. As a generalization of this fact, we show that when $\pd_R(M)< \infty$, $H^i_{\fa}(M, N) =0$ for all $i> \pd_R(M)+ \max\{\cd_{\fa}(\mathrm{Ext}^{i}_{R}(M,
N))\mid i\in \N_0\} =: c$ and that $H^c_{\fa}(M, N)= 0$ for all $n\ll 0$ or, in a special case, $H^c_{\fa}(M, N)_n \neq0$ for all $n\ll 0$.
 To this end we need to provide some lemmas.

\begin{lem}\label{2.4}%-----------------------------------------------------------------------------------------------------------------------------
Let $E $ be an injective $R$-module. Then, $H^{i}_{\fa}
(\mathrm{Hom}_{R}(M, E))=0$ for all $i\in \mathbb{N}$.
\end{lem}
\begin{proof}
Note that if $P_{\bullet}^{M}$ is a free resolution of $M$,
then $\mathrm{Hom}_{R}(P_{\bullet}^{M}, E)$ is an injective
resolution of $\mathrm{Hom}_{R}(M, E)$. Now, the assertion follows
from the following isomorphisms
\[H^{i}_{\fa} (\mathrm{Hom}_{R}(M, E))\cong
H^{i}(\mathrm{Hom}_{R}(P_{\bullet}^{M}, \Gamma _{\fa }(E)))\cong
\mathrm{Ext}^{i}_{R}(M, \Gamma _{\fa }(E)),\]  in conjunction with the
fact that $\Gamma _{\fa }(E)$ is an injective $R$-module.
\end{proof}

\begin{lem}\label{cd}%-----------------------------------------------------------------------------------------------------------------------------
Let $p:= \pd_R(M)< \infty$ and $c:=\mathrm{max} \{ \mathrm{cd}_{\fa}(\mathrm{Ext}^{i}_{R}(M,
N))| 0\leq i\leq p\}$. Then the following statements hold:
\\$(\mathrm{i})$ $H^{i}_{\fa} (M, N)=0$ for all $i>p+c$;
\\$(\mathrm{ii})$ $H^{p+c}_{\fa} (M, N)\cong H^{c}_{\fa} (\mathrm{Ext}^{p}_{R}(M,
N))$;
\\ $(\mathrm{iii})$ $\mathrm{cd}_{\fa}(M, N)=p+c$ if and only if
$c=\mathrm{cd}_{\fa}(\mathrm{Ext}^{p}_{R}(M, N))$.
\end{lem}
\begin{proof}
Using \cite[11.38]{r} and \ref{2.4}, there exists a Grothendieck's spectral sequence
\[E_{2}^{i, j}=H^{i}_{\fa}(\mathrm{Ext}^{j}_{R}(M, N))\underset{i}{\Rightarrow}H^{i+ j}_{\fa} (M, N).\]
Since $E_{2}^{i, j}=0$ for
$i>c$ or $j>p$, we have $H^{i}_{\fa} (M, N)= 0$ for all $i> p+ c$ and that $$H^{p+c}_{\fa} (M, N)\cong E_{\infty}^{c, p}\cong E_{2}^{c, p}= H^{c}_{\fa}(\mathrm{Ext}^{p}_{R}(M, N)).$$ Therefore, $\mathrm{cd}_{\fa}(M, N)=p+c$ if and only if
$c=\mathrm{cd}_{\fa}(\mathrm{Ext}^{p}_{R}(M, N))$.
\end{proof}

\begin{defen}%-----------------------------------------------------------------------------------------------------------------------------
Let $\fb$ be an ideal of $R$. Then $M$ is said to be relative Cohen-Macaulay with respect to $\fb$ if $\cd_{\fb}(M)= \grade(\fb, M)$.(see \cite{jr})
\end{defen}

In \cite{z} it is shown that $H ^{i}_{R_+}(M, N)= 0$ for all $i\geq \cd_{R_+}(N)+ \pd_R(M)$ and that $H ^{\cd_{R_+}(N)+ \pd_R(M)}_{R_+}(M, N)$ is tame. In the next theorem, using the notations in \ref{cd}, we are going to study whether $H ^{c+ p}_{\fa}(M, N)$ is tame or not.

\begin{thm}\label{tame}%-----------------------------------------------------------------------------------------------------------------------------
Let the situations be as in \ref{cd}. Then one of the followings holds:
\\
(\textrm{i}) $H ^{c+ p}_{\fa}(M, N)= 0$;
\\(\textrm{ii}) $H ^{c+ p}_{\fa}(M, N)_n=0$ for all $n\ll 0$;
\\(\textrm{iii}) if $\Ext^p_R(M, N)$ is relative Cohen-Macaulay with respect to $R_+$ with  $\cd_{R_+}(\Ext^p_R(M, N))> 0$ then $H ^{c+ p}_{\fa}(M, N)_n\neq 0$ for all $n\ll 0$.
\end{thm}
\begin{proof}
For simplicity set $X:= \Ext^p_R(M, N)$. Using \ref{cd}, if $\cd_{\fa}(X)< c$ then $H ^{c+ p}_{\fa}(M, N)= 0$. So, let $\cd_{\fa}(X)= c$. Now, if $c':= \cd_{R_+}(X)= 0$ then
using the Noetherian property of $X$, we have $H^{c+ p}_{\fa}(M, N)_n\cong H^{c}_{\fa} (X)_n\cong H^{c}_{\fa_0}(X_n)= 0$ for all $i$ and all $n\ll 0$. So, let $c'> 0$ and assume, in addition, that $X$ is relative Cohen-Macaulay with respect to $R_+$. Therefore, in view of the convergence of spectral sequences
$$E_2^{i, j}= H^{i}_{\fa_0R} (H^{j}_{R_+} (X))\underset{i}{\Rightarrow}H^{i+ j}_{\fa} (X)$$
and the fact that $E_2^{i, j}= 0$ for all $i\in \N_0$ and all $j\neq c'$, we have the following isomorphism of graded modules $$\hspace{3cm}H^{i}_{\fa_0R} (H^{c'}_{R_+} (X))\cong H^{i+ c'}_{\fa} (X) \ \ \ \mbox{ for all} \ \ \ i\in \N_0. \ \hspace{5cm} (1) $$
Replacing $R_0$ with its faithfull flat extension $R_0[X]_{\m_0[X]}$, we may assume that the residue field $R_0/\m_0$ is infinite. This in conjunction with  \cite[2.3(a)]{b} and \cite[1.5.12]{bh} implies that
$\exists x\in R_1$ such that it is a non-zerodivisor on $X$ and that $\cd(R_+, X/ xX)= \cd_{R_+}(X)- 1= c'- 1$. So, there exist exact sequences
$$\hspace{2.3cm}H^{c'- 1}_{R_+}(X/ xX)_n \stackrel{f_n}{\longrightarrow} H^{c'}_{R_+}(X)_{n- 1}\stackrel{.x}{\longrightarrow} H^{c'}_{R_+}(X)_{n}\longrightarrow 0 \hspace{4cm}(2)$$
for all $n\in \Z$.
On the other hand,
$$c= \cd_{\fa}(X)\geq \grade(\fa, X)\geq \grade(R_+, X)=  \cd_{R_+}(X)= c'.$$
So, using $(1)$, $H^{c- c'}_{\fa_0R} (H^{c'}_{R_+} (X))\cong H^{c}_{\fa} (X)\neq 0$, which implies that $ \cd_{\fa_0R}(H^{c'}_{R_+} (X))= c- c'$.
Hence, in view of \cite[2.2]{dnt} and the fact that $\Supp_{R_0}(\im(f_n))\subseteq \Supp_{R_0}(H^{c'}_{R_+}(X)_{n- 1})$, we have
$$\cd_{\fa_0}(\im(f_n))\leq \cd_{\fa_0}(H^{c'}_{R_+}(X)_{n- 1})\leq \cd_{\fa_0R}(H^{c'}_{R_+}(X))= c- c'.$$
So, using $(2)$, we get the following exact sequence
$$\hspace{3.3cm}H^{c- c'}_{\fa_0}( H^{c'}_{R_+}(X)_{n- 1})\longrightarrow H^{c- c'}_{\fa_0}( H^{c'}_{R_+}(X)_{n})\longrightarrow 0. \hspace{4cm}(3)$$
Since $ \cd_{\fa_0R}(H^{c'}_{R_+}(X))= c- c'$, hence $\exists n_0\in \Z$ such that $H^{c- c'}_{\fa_0}( H^{c'}_{R_+}(X)_{n_0})\neq 0$. This, in conjunction with \ref{cd}(ii), $(1)$ and  $(3)$ implies that
$$H^{c+ p}_{\fa} (M, N)_n\cong H^{c}_{\fa} (X)_n\cong H^{c- c'}_{\fa_0}( H^{c'}_{R_+}(X)_{n})\neq 0$$
   for all $n\leq n_0$.
\end{proof}

The following corollary implies a tameness property of the ordinary local cohomology modules at the top level. It is directly concluded by \ref{tame}(iii).
\begin{cor}%-----------------------------------------------------------------------------------------------------------------------------
 If $\Gamma_{R_+}(N)\neq N$ and $N$ is relative Cohen-Macaulay with respect to $R_+$, then
$H^{\cd_{\fa}(N)}_{\fa}(N)_n\neq 0$ for all $n\ll 0$.
\end{cor}

\begin{defins}%-----------------------------------------------------------------------------------------------------------------------------
(i) Following \cite{m}, we call a graded $R$-module $X$ to be \emph{finitely graded},
if $X_n=0$ for all but finitely many $n\in \Z$, where $X_n$ denotes the
$n$-th graded piece of $X$.

Also, we set
$$\upsilon_{\fa}(M, N):= \sup\{k\in\N_0\mid H_{\fa}^i(M,N)\ \mbox{is
finitely graded for all}\ i<k\}.$$

(ii) The $R_+$ finiteness dimension of $M$ and $N$ with respect to $\fa$, is defined to be
$$f_{\fa}^{R_+}(M,N):=\sup\{k\in\N_0\mid R_+\subseteq\sqrt{0:_R
H_{\fa}^i(M,N)}\ \mbox{for all}\ i<k\}.$$

%(\cite{ADT}.)
\end{defins}

In the rest of this section we are going to show that
$\upsilon_{\fa}(M, N)= f_{\fa}^{R_+}(M,N)$, in the case where $\pd_R(M)< \infty$. Then, as a consequence, we can prove that there exists a finite subset $X$ of $\Spec(R_0)$ such that $\Ass_{R_0}(H_{\fa}^{f_{\fa}^{R_+}(M,N)}(M,N)_n)= X$ for all $n\ll 0$, which generalize \cite[3.4]{jz}.

To this end, we need to remind some results from \cite{m}.

\begin{lem}\label{m1}%-----------------------------------------------------------------------------------------------------------------------------
Let $X$ be a finitely graded $R$-module. Then $R_+\subseteq \sqrt{0:_R X}$. Furthermore, if $X$ is finitely generated, then the converse is true.
\end{lem}
\begin{proof}
See \cite[2.1]{m}.
\end{proof}

\begin{lem}\label{mm}%-----------------------------------------------------------------------------------------------------------------------------
Let $X$ be a finitely graded $R$-module. Then $H^i_{\fb}(X)$ is finitely graded for all $i\in \N_0$ and all homogenous ideal $\fb$ of $R$.
\end{lem}
\begin{proof}
See \cite[2.2]{m}.
\end{proof}

The following lemma is a generalization of the above one. We need it to prove next theorem.

\begin{lem}\label{fg}%-----------------------------------------------------------------------------------------------------------------------------
Let $p:= \pd_R(M)< \infty$ and $R_+\subseteq\sqrt{0:_RN}$. Then $H_{\fa}^i(M,N)$ is finitely graded
for all $i\in \N_0$.
\end{lem}
\begin{proof}
Let $i\in \N_0$. We use induction on $p:=\pd_R (M)$ to prove the claim. First, suppose $p=0$. Then
there exist a positive integer $n$ and a finitely generated graded
$R$-module $M'$ such that $M\oplus M'\cong R^n$. Thus $H_{\fa}^i(M,N)\oplus
H_{\fa}^i(M',N)\cong H_{\fa}^i(R^n,N)\cong(H_{\fa}^i(R,N))^n=(H_{\fa}^i(N))^n$. Therefore, in the case $p= 0$, the result follows from \ref{m1} and \ref{mm}. Now, let $p>0$ and assume that the result has been proved for every
finitely generated graded $R$-module $M'$ with finite projective
dimension $p-1$. So that, there exist a positive integer $n$, a
finitely generated graded $R$-module $M'$ with projective dimension
$p-1$ and a short exact sequence $0\rightarrow M'\rightarrow
R^n\rightarrow M\rightarrow0$. It yields to the exact sequence
$H_{\fa}^{i-1}(M',N)\rightarrow H_{\fa}^i(M,N)\rightarrow
H_{\fa}^i(R^n,N)=(H_{\fa}^i(N))^n$. Now, the inductive
hypothesis together with \ref{m1} and \ref{mm} completes the proof.
\end{proof}

\begin{thm}\label{vf}%-----------------------------------------------------------------------------------------------------------------------------
Let $\pd_R(M)< \infty$. Then $\upsilon_{\fa}(M, N)= f_{\fa}^{R_+}(M,N).$
\end{thm}

\begin{proof}
By \ref{m1}, $\upsilon_{\fa}(M,N)\leq f_{\fa}^{R_+}(M,N)$. To prove $\upsilon_{\fa}(M,N)\geq
f_{\fa}^{R_+}(M,N)$, we make some reductions. First, using previous
lemma and the inequality $\upsilon_{\fa}(M,N)\leq f_{\fa}^{R_+}(M,N)$, we may
assume that $R_+\not\subseteq\sqrt{0:_RN}$.

Let $x\in R_+\backslash \bigcup_{\p\in \Ass_R(N)\backslash \Var(R_+)} \p$. Then, in view of \ref{m1}, it is straightforward to see that there exists $m\in \N_0$ such that $0:_M x^m= \Gamma_{(x)}(M)= \Gamma_{R_+}(M)$ is
finitely graded. Now, using the
exact sequence $0\rightarrow 0:_N x^m\rightarrow N\rightarrow
N/0:_N x^m\rightarrow 0$, \ref{seq}(iv) and \ref{fg}, we deduce that
$\upsilon_{\fa}(M,N)=\upsilon_{\fa}(M,N/0:_N x^m)$ and $f_{\fa}^{R_+}(M,N)=f_{\fa}^{R_+}(M,N/0:_Nx^m)$.
Thus, replacing $N$ with $N/ 0:_N x^m$, we may assume that $x$ is non-zerodivisor on $N$. Also, by
definition of $f_{\fa}^{R_+}(M,N)$, there exists an integer $l\geq1$
such that $x^lH_{\fa}^i(M,N)=0$ for all $i<f_{\fa}^{R_+}(M,N)$. So,
replacing $x$ with $x^l$, we may also assume that $xH_{\fa}^i(M,N)=0$ for all
$i<f_{\fa}^{R_+}(M,N)$.

Now, to prove $\upsilon_{\fa}(M,N)\geq f_{\fa}^{R_+}(M,N)$, we show, by induction on
$k$, that $0\leq k\leq \upsilon_{\fa}(M,N)$ whenever $0\leq k\leq
f_{\fa}^{R_+}(M,N)$. To do this, let $t=\deg(x)$. Then the exact sequence $0\rightarrow N\stackrel{.x}{\rightarrow}
N(t)\rightarrow (N/xN)(t)\rightarrow 0$, in conjunction with \ref{seq}(iv), yields the long exact
sequence
%\begin{equation}
$$H_{\fa}^i(M,N)(t)\rightarrow H_{\fa}^i(M,N/xN)(t) \rightarrow
H_{\fa}^{i+1}(M,N)\stackrel{.x}{\rightarrow}H_{\fa}^{i+1}(M,N)(t)$$
%\end{equation}
for all $i\geq0$. This gives $(k-1\leq) f_{\fa}^{R_+}(M,N)-1\leq
f_{\fa}^{R_+}(M,N/xN)$. By inductive hypothesis, $k-1\leq
\upsilon_{\fa}(M,N/xN)$; so, by definition of $\upsilon_{\fa}(M, N/xN)$,
$H_{\fa}^{k-2}(M,N/xN)$ is finitely graded. Hence,
$0\rightarrow
H_{\fa}^{k-1}(M,N)_n\stackrel{.x}{\rightarrow}H_{\fa}^{k-1}(M,N)_{n+t}$ is
injective for all but finitely many integers $n$. Therefore, the
assumption $xH_{\fa}^{k-1}(M,N)=0$ shows that $H_{\fa}^{k-1}(M,N)$ is
finitely graded. This leads us the inequality $k\leq \upsilon_{\fa}(M,N)$.
\end{proof}

The following lemma, which is used to prove next theorem, can be proved with similar argument as used in \cite[3.3]{jz} in conjunction with \cite[3.2]{kh}.
\begin{lem}%-----------------------------------------------------------------------------------------------------------------------------
Let $i\in \N_0$ and $n_0\in \Z$ be such that, for all $j< i$, $H_{\fa}^{j}(M,N)_n$ is a finitely generated $R_0$-module for all $n\leq n_0$. Then, for all $n\leq n_0$ and any finitely generated submodule $L$ of $H_{\fa}^{i}(M,N)_n$, the set $\Ass_{R_0}(H_{\fa}^{i}(M,N)_n/L)$ is finite.
\end{lem}

\begin{rem}%-----------------------------------------------------------------------------------------------------------------------------
Note that, in view of \cite[2.2(iv)]{adt} and \cite[9.1.8]{b},
\[f_{\fa}^{R_+}(M, N)\geq f_{\fa}^{R_+}(N)\geq f_{R_+}^{R_+}(N)= f_{R_+}(N).\]

Also, when $\pd_R(M)< \infty$, using \ref{vf}, one has $f_{\fa}^{R_+}(M,N)\leq g_{\fa}(M,N)$.
\end{rem}
 In view of the above remark, next theorems recover \cite[3.4 and 3.6]{jz} and \cite[5.6]{bhe}.
\begin{thm}\label{ass}%-----------------------------------------------------------------------------------------------------------------------------
Let $(R_0, \m_0)$ be local and $\pd_R(M)< \infty$. Then for all $i\leq f_{\fa}^{R_+}(M,N)$ there exists a nonempty finite subset $X$ of $\Spec(R_0)$ such that $\Ass_{R_0}(H_{\fa}^{i}(M,N)_n)= X$ for all $n\ll 0$.
\end{thm}
\begin{proof}
For $i< f_{\fa}^{R_+}(M,N)$, the result is clear using \ref{vf}. And when $i= f_{\fa}^{R_+}(M,N)$, using \ref{fg} and previous lemma, one can prove the claim by employing the method of proof which used in  \cite[3.4]{jz}.
\end{proof}

\begin{lem}%-----------------------------------------------------------------------------------------------------------------------------
Let $(R_0, \m_0)$ be local. Then for all $i\leq g_{\fa}(M,N)$, $\Gamma_{\m_0R}(H_{\fa}^{i}(M,N))$ is tame.
\end{lem}
\begin{proof}
One can use, \cite[2.2]{z2} and the argument used in \cite[3.5]{jz} to prove the claim.
\end{proof}

\begin{thm}%-----------------------------------------------------------------------------------------------------------------------------
Let the situation be as in \ref{ass}. Then for all $i\leq g_{\fa}(M,N)$ there exists a nonempty finite subset $X$ of $\Spec(R_0)$ such that $\Ass_{R_0}(H_{\fa}^{i}(M,N)_n)= X$ for all $n\ll 0$
\end{thm}
\begin{proof}
The claim can be proved with a slight modification of \cite[3. 6]{jz}.
\end{proof}

In the rest of this section, we are going to study $H_{\fa}^{i}(M,N)$ in the case where $R_+$ is principal.

\begin{thm}%-----------------------------------------------------------------------------------------------------------------------------
Let $(R_0, \m_0)$ be local and $y\in R_1$. Then for any ideal $\fa_0\subseteq \m_0$ and all $i\leq g_{yR}(M, N)$, $\exists n_0\in \Z$ such that $H_{yR+ \fa_0}^{i}(M,N)_{n_0}\cong H_{yR+ \fa_0}^{i}(M,N)_{n}$ for all $n\leq n_0$.
\end{thm}

\begin{proof}
Let $\fa_0= (x_1,\cdots, x_t)$. We use induction on $t$. Using \ref{seq}(iii), $\exists n_0\in \Z$ such that $H^i_{yR}(M, \Gamma_{yR}(N))_n\cong \Ext^i_R(M, \Gamma_{yR}(N))_n= 0$ for all $n\leq n_0$. So the exact sequence $0\longrightarrow \Gamma_{yR}(N)\longrightarrow N\longrightarrow N/ \Gamma_{yR}(N)\longrightarrow 0$ implies that $H_{yR}^{i}(M,N)_{n}\cong H_{yR}^{i}(M, N/ \Gamma_{yR}(N))_{n}$ for all $n\leq n_0$. Therefore, we can assume that $y$ is a non-zero divisor on $N$. Now, in view of \ref{seq}(iii) and the exact sequence $$H^{i- 1}_{yR}(M, N/ yN)_n\longrightarrow H^{i}_{yR}(M, N)_{n- 1}\stackrel{.y}{\longrightarrow} H^{i}_{yR}(M, N)_n \longrightarrow H^{i}_{yR}(M, N/ yN)_n,$$
for all $i\in \N_0$ there exists $n_0\in \Z$ such that $H_{yR }^{i}(M,N)_{n_0}\cong H_{yR }^{i}(M,N)_{n}$ for all $n\leq n_0$. This proves the claim in the case $t= 0$.

Now let $t> 0$, $i\leq g_{yR}(M, N)$ and assume that the result has been proved for smaller values of $t$. Using \ref{g}, $i- 1< g_{yR+ (x_1,\cdots, x_{t- 1})}(M, N)$. Therefore $(H^{i-1}_{yR+ (x_1,\cdots, x_{t- 1})}(M, N)_n)_{x_t}= 0$ for all $n\ll 0$. Also, in view of inductive hypothesis $\exists n_0\in \Z$ such that $H^i_{yR+ (x_1,\cdots, x_{t- 1})}(M, N)_{n_0}\cong H^i_{yR+ (x_1,\cdots, x_{t- 1})}(M, N)_n$ for all $n\leq n_0$. Therefore, using \ref{seq}(i), one can see that there exists a completed commutative diagram with exact rows
\\
\\
$$0\longrightarrow H_{yR+ \fa_0}^{i}(M,N)_n\longrightarrow H^i_{yR+ (x_1,\cdots, x_{t- 1})}(M, N)_n\longrightarrow (H^i_{yR+ (x_1,\cdots, x_{t- 1})}(M, N)_n)_{x_t}$$
$
\hspace{3cm}  \downarrow \hspace{4cm} \downarrow \hspace{5cm} \downarrow
$
$$0\longrightarrow H_{yR+ \fa_0}^{i}(M,N)_{n_0}\longrightarrow H^i_{yR+ (x_1,\cdots, x_{t- 1})}(M, N)_{n_0}\longrightarrow (H^i_{yR+ (x_1,\cdots, x_{t- 1})}(M, N)_{n_0})_{x_t}$$

such that the last two vertical homomorphisms are isomorphism. Now, five lemma implies that $H_{yR+ \fa_0}^{i}(M,N)_n\cong H_{yR+ \fa_0}^{i}(M,N)_{n_0}$ for all $n\leq n_0$ and the result follows by induction.
\end{proof}

\begin{cor}%-----------------------------------------------------------------------------------------------------------------------------
Let $(R_0, \m_0)$ be local, $R_+$ be principal and $\pd_R(M)< \infty$. Then for any $i\leq g_{R_+}(M, N)$ and ideal $\fa_0\subseteq \m_0$ of $R_0$, there are only a finite number of non-isomorph graded components of $H_{\fa_0 + R_+}^{i}(M,N)$.
\end{cor}

\begin{proof}
In view of \cite[2.5]{y}, $H_{\fa_0 + R_+}^{i}(M,N)= 0$ for $i\gg 0$. Now, the result follows from the previous Theorem and \ref{gen}(i).
\end{proof}
%%%%%%%%%%%%%%%%%%%%%%%%%%%%%%%%%%%%%%%%%%%%%%%%%%%%%%%%%%%%%%%%%%%%%%%%%%%%%%%%%%%%%%%%%%%%%%%%%%%%%%%%%%%%%%%%%%%%%%%%%%%%%%%%%%%%%%%%%%%%%%%%
\section{Tameness and Artinianness at non-minimax levels}

In this section we are going to study tameness and Artinian property of some submodules and quotient modules of $H^i_{\fa}(M, N)$ at \emph{non-minimax} levels.

In the rest of the paper, $(R_0, \m_0)$ is a local ring and $\fb_0$ will denote an ideal of $R_0$ such that $\sqrt{\fa_0+ \fb_0}= \m_0$.
\begin{defin}%-----------------------------------------------------------------------------------------------------------------------------
A graded  $R$-module $X$ is said to be $^*$minimax, if it has a finitely generated graded submodule $Y$, such that $X/ Y$  is an Artinian $R$-module.
\end{defin}

By the following lemma, any graded submodule and any homogeneous homomorphic image of a $^*$minimax module is $^*$minimax.
%In other words the category of all $^*$minimax modules is a Serre category.

\begin{lem}%-----------------------------------------------------------------------------------------------------------------------------
Let $0\to X\to Y\to Z\to 0$  be an exact sequence of graded $R$-modules and $R$-homomorphisms. Then $Y$ is $^*$minimax if and only if $X$ and $Z$ are $^*$minimax.
\end{lem}
\begin{proof}
See \cite[2.1]{bn}.
\end{proof}
The following lemma is needed to prove most of the results in this section.

\begin{lem}\label{art}%-----------------------------------------------------------------------------------------------------------------------------
Let $X$  be a graded $^*$minimax $R$-module. If $X$  is  $\fa$-torsion, then for all $j\in \N_0$,  $\Tor_j^{R_0}(R_0/ \fb_0, X)$ and  $H^j_{\fb_0R}(X)$ are Artinian $R$-modules.
\end{lem}
\begin{proof}
There exists a finitely generated graded submodule  $Y$ of $X$, such that  $X/ Y$ is Artinian. Now, the exact sequence
 $0\to Y\to X\to X/ Y\to 0$ induces two long exact sequences
 $$\cdots\to \Tor_j^{R_0}(R_0/ \fb_0, Y)\to \Tor_j^{R_0}(R_0/ \fb_0, X)\to \Tor_j^{R_0}(R_0/ \fb_0, X/ Y)\to \Tor_{j-1}^{R_0}(R_0/ \fb_0, Y)\to \cdots$$
 and
 $$\cdots\to H^j_{\fb_0R}(Y)\to H^j_{\fb_0R}(X)\to H^j_{\fb_0R}(X/ Y)\to H^{j+1}_{\fb_0R}(Y)\to \cdots.$$

As $Y$  is a finitely generated  $\fa$-torsion $R$-module, it is easy to see that for all $j\in \N_0$, $\Tor_j^{R_0}(R_0/ \fb_0, Y)$  and  $H^j_{\fb_0R}(Y)$ are Artinian $R$-modules. Also, by \cite[2.2]{bft},  $H^j_{\fb_0R}(X/ Y)$ and  $\Tor_j^{R_0}(R_0/ \fb_0, X/ Y)$ are Artinian. Now, the result follows from the above exact sequences.
\end{proof}

\begin{nota rem}%-----------------------------------------------------------------------------------------------------------------------------
For any graded  $R$-modules $X$  and $Y$  set
$$t_{\fa}(X, Y):= \inf\{i\in \N_0| H^i_{\fa}(X, Y) \  \mbox{is not $^*$minimax}\}$$
and
$$ s_{\fa}(X, Y):= \sup\{i\in \N_0| H^i_{\fa}(X, Y) \  \mbox{is not $^*$minimax}\}.$$
Since any Noetherian (or Artinian) graded module is $^*$minimax, so $f_{\fa}(X, Y)\leq t_{\fa}(X, Y)$, where $f_{\fa}(X, Y):= \inf\{i\in \N_0\mid H^i_{\fa}(M, N) \ \mbox{is not finitely generated} \}$ is the finiteness dimension of $M$ and $N$ with respect to $\fa$.
\end{nota rem}

Now, we prove a lemma which will be used to prove next proposition.

\begin{lem}\label{red}%-----------------------------------------------------------------------------------------------------------------------------
Using the above notations, the following statements hold.
\\
\textrm{(i)} $t_{\fa}(M, N/ \Gamma_{\fb_0R}(N))= t_{\fa}(M, N)$ and $s_{\fa}(M, N/ \Gamma_{\fb_0R}(N))= s_{\fa}(M, N)$;
\\
\textrm{(ii)} for any  $i\in \N_0$, $R_0/ \fb_0\otimes_{R_0}H^i_{\fa}(M, N)$  is an Artinian $R$-module if and only if $R_0/ \fb_0\otimes_{R_0}H^i_{\fa}(M, N/ \Gamma_{\fb_0R}(N))$  is.
\end{lem}
\begin{proof}
The exact sequence  $0\to \Gamma_{\fb_0R}(N)\to N\to N/ \Gamma_{\fb_0R}(N)\to 0$ induces the long exact sequence
$$H^i_{\fa}(M, \Gamma_{\fb_0R}(N))\to H^i_{\fa}(M, N) \stackrel{\eta}{\to} H^i_{\fa}(M, N/ \Gamma_{\fb_0R}(N))\to H^{i+1}_{\fa}(M, \Gamma_{\fb_0R}(N)) $$
of generalized local cohomology modules.
As  $\sqrt{\fb_0+\fa}= \m_0+ R_+=\m$ is the graded maximal ideal of $R$, in view of \cite[2.2]{z2}, $H^i_{\fa}(M, \Gamma_{\fb_0R}(N))\cong H^i_{\m}(M, \Gamma_{\fb_0R}(N))$ is Artinian for all  $i\in \N_0$. Thus the above exact sequence implies that  $t_{\fa}(M, N/ \Gamma_{\fb_0R}(N))= t_{\fa}(M, N)$ and  $s_{\fa}(M, N/ \Gamma_{\fb_0R}(N))= s_{\fa}(M, N)$, which proves (i), and that $\ker(\eta)$  and $\coker(\eta)$  are Artinian $\fa$-torsion $R$-modules. Now, consider exact sequences
$$0\to \ker(\eta)\to H^i_{\fa}(M, N)\to \im(\eta)\to 0\ \ \mbox{
and}\ \ \
0\to \im(\eta)\to H^i_{\fa}(M, N/\Gamma_{\fb_0R}(N))\to \coker(\eta)\to 0$$
to get the following exact sequences
$$R_0/ \fb_0\otimes_{R_0} \ker(\eta)\to R_0/ \fb_0\otimes_{R_0}H^i_{\fa}(M, N) \to R_0/ \fb_0\otimes_{R_0}\im(\eta)\to 0$$
and
$$\Tor_1^R(R_0/ \fb_0, \coker(\eta))\to R_0/ \fb_0\otimes_{R_0}\im(\eta)\to R_0/ \fb_0\otimes_{R_0}H^i_{\fa}(M, N/ \Gamma_{\fb_0R}(N))\to R_0/ \fb_0\otimes_{R_0} \coker(\eta)\to 0.$$
By Lemma \ref{art} all ended modules of these sequences are Artinian. So, $R_0/ \fb_0\otimes_{R_0}H^i_{\fa}(M, N)$  is Artinian if and only if   $R_0/ \fb_0\otimes_{R_0}\im(\eta)$ is Artinian and this is, if and only if  $H^i_{\fa}(M, N/ \Gamma_{\fb_0R}(N))$ is Artinian.
\end{proof}
%The next proposition shows that $t_{\fa}(M, N)$ is an element of $\Delta_1$.

\begin{prop}\label{gam}%-----------------------------------------------------------------------------------------------------------------------------
Let $j\leq t_{\fa}(M, N)$. Then
%\\
 %(i)  $R_0/ \fb_0\otimes_{R_0}H^i_{\fa}(M, N)$ is an Artinian $R$-module; and
 %\\
% (ii)
   $H^i_{\fb_0R}(H^j_{\fa}(M, N))$ is  Artinian for $i=0, 1$.
\end{prop}

\begin{proof}
%(i) Using \ref{art}, the result is clear for all $i< t_{\fa}(M, N):= t$. So, it is enough to show that $R_0/ \fb_0\otimes_{R_0}H^t_{\fa}(M, N)$ is Artinian. By lemma \ref{red}, we can assume %that  $N$ is  $\Gamma_{\fb_0}$-torsion free. Thus, there exists an element $x\in \fb_0$, such that $x$  is a non-zero divisor on  $N$. Now the exact sequence $0\to N \stackrel{.x}{\to} N \to %N/ xN \to 0$  induces the long exact sequence
%$$\cdots \to H^{i- 1}_{\fa}(M, N)\to H^{i- 1}_{\fa}(M, N/ xN)\to H^{i}_{\fa}(M, N)\stackrel{.x}{\to}H^{i}_{\fa}(M, N)\to \cdots.$$
%If $i< t$,  then $H^{i- 1}_{\fa}(M, N/ xN)$  is $^*$minimax, by the above sequence. So, $t_{\fa}(M, N/ xN)\geq t- 1$. Also, when $i= t$, the above long exact sequence induces an exact %sequence
%$$R_0/ \fb_0\otimes_{R_0} H^{t- 1}_{\fa}(M, N/ xN)\to R_0/ \fb_0\otimes_{R_0} H^{t}_{\fa}(M, N)\stackrel{.x}{\to} R_0/ \fb_0\otimes_{R_0} xH^{t}_{\fa}(M, N). $$
%As  $x\in \fb_0$, the multiplication map $.x$ is zero. So, $R_0/ \fb_0\otimes_{R_0} H^{t}_{\fa}(M, N)$ is a homomorphic image of
%$R_0/ \fb_0\otimes_{R_0} H^{t- 1}_{\fa}(M, N/ xN)$. Now, the claim (i)  follows by induction on $t$.
%\\
%(ii)
 For all $i< t:= t_{\fa}(M, N)$, the result is clear by lemma \ref{art}. So, let $i= t$ and consider the following spectral sequence
$$E_2^{i, j}:= H^i_{\fb_0R}(H^j_{\fa}(M, N))\underset{i}{\Rightarrow} H^{i+ j}_{\m}(M, N).$$
Note that $E_2^{i, j}= 0$  for all $i< 0$ and let $r\geq 2$. So, if  $i=0, 1$, then the sequence
$$0\to E_{r+ 1}^{i, t}\to E_{r}^{i, t}\stackrel{d_r^{i, t}}{\to} E_{r}^{i+ r, t- r+ 1}$$
 is exact.
In view of lemma \ref{art}, as a subquotient of $E_2^{i+ r, t- r+ 1}= H^{i+ r}_{\fb_0R}(H^{t- r+ 1}_{\fa}(M, N))$, $E_{r}^{i+ r, t- r+ 1}$ is Artinian. Now, let $r_0\geq 2$ be an integer such that
$E_{r_0+ 1}^{i, t}= E_{r_0+ 2}^{i, t}= \cdots =E_{\infty}^{i, t}$. Then, using the Artinianness of $H^{i+ t}_{\m}(M, N)$ and the fact that $E_{r_0+ 1}^{i, t}$ is a subquotient of $H^{i+ t}_{\m}(M, N)$, $E_{r_0+ 1}^{i, t}$ is Artinian. Thus, from the above exact sequence, it follows that $E_{2}^{i, t}= H^i_{\fb_0R}(H^t_{\fa}(M, N))$ is Artinian.

 \end{proof}

Now, using the above theorem, we can prove a stability result of the components of $H^{i}_{R_+}(M, N)_n$, when $n\to -\infty$.

\begin{thm}%-----------------------------------------------------------------------------------------------------------------------------
Let $i\leq t_{R_+}(M, N)$. Then,
\\
 %(i) the set  $\Ass_{R_0}(H^{i}_{R_+}(M, N)_n)$ is asymptotically stable for $n\to -\infty$;\\
 (i) there is a numerical polynomial $P(x)\in \mathbb{Q}[x]$  of degree less  than $i$, such that\\
 $\length_{R_0}(\Gamma_{\m_0}(H^{i}_{R_+}(M, N)_n))= P(n)$ for all $n\ll 0$, and\\
 (ii) there is a numerical polynomial  $\acute{P}(x)\in \mathbb{Q}[x]$  of degree less  than $i$, such that\\
$\length_{R_0}(0:_{H^{i}_{R_+}(M, N)_n} \m_0)= \acute{P}(n)$ for all $n\ll 0$.
\end{thm}

\begin{proof}
(i) Let $i\leq t_{R_+}(M, N)$.
Using \ref{gam}, $\Gamma_{\m_0R}(H^i_{R_+}(M, N))$  is an  Artinian $R$-module. So, by \cite{k} there exists a  polynomial $P(x)\in \mathbb{Q}[x]$  such that  $\length_{R_0}(\Gamma_{\m_0}(H^{i}_{R_+}(M, N)_n))= P(n)$ for all $n\ll 0$.
It remains to show that $\deg(P(x))\leq i$. Note that on use of a standard reduction argument, replacing $R_0$ with its faithful flat extension $R_0[X]_{\m_0[X]}$, we can assume that the residue field $R_0/ \m_0$ is infinite.

Also, in view of \ref{seq}(iii), $H^i_{R_+}(M, N)_n\cong H^i_{R_+}(M, N/\Gamma_{R_+}(N))_n$ for all $n\ll 0$. Hence, we may replace $N$ with $N/\Gamma_{R_+}(N)$ and assume that there exists $x\in R_1$ which is a non-zerodivisor on $N$.
Now, consider the exact sequence
$0\to N(-1)\stackrel{.x}{\to} N\to N/ xN\to 0$   to get the long exact sequence
$$H^{i-1}_{R_+}(M, N)_n \to H^{i- 1}_{R_+}(M, N/ xN)_n\stackrel{f_n}{\to}  H^i_{R_+}(M, N)_{n- 1}\stackrel{.x}{\to} H^i_{R_+}(M, N)_n,  $$
of $R_0$-modules. Which induces the exact sequence
$$\hspace{3cm} 0\to \Gamma_{\m_0}(\im(f_n))\to \Gamma_{\m_0}(H^i_{R_+}(M, N)_{n- 1})\to \Gamma_{\m_0}(H^i_{R_+}(M, N)_n), \hspace{1.5cm}(1)$$
for all $n\in \Z$ and, also, shows that $t_{R_+}(M, N/xN)\geq t_{R_+}(M, N)- 1\geq i- 1$.
The fact that $H^{i-1}_{R_+}(M, N)$ is a $^*$minimax $R$-module implies $H^{i-1}_{R_+}(M, N)_n$ is an Artinian $R_0$-module for all $n\ll 0$. So, there exists epimorphisms
$$\Gamma_{\m_0}(H^{i-1}_{R_+}(M, N/xN)_n)\to \Gamma_{\m_0}(\im(f_n))\to 0 \ \ \ \ \ \mbox{for all } \ n\ll 0.$$

This in conjunction with $(1)$ shows that
$$\length_{R_0}(\Gamma_{\m_0}(H^i_{R_+}(M, N)_{n-1}))- \length_{R_0}(\Gamma_{\m_0}(H^i_{R_+}(M, N)_n)\leq \length_{R_0}(\Gamma_{\m_0}(H^{i- 1}_{R_+}(M, N/ xN)_{n}))$$
for all $n\ll 0$.
This allows to conclude by induction on $i\leq t_{R_+}(M, N)$.
\\
(ii)
 As a submodule of  $\Gamma_{\m_0R}(H^i_{R_+}(M, N))$, the $R$-module $0:_{H^{i}_{R_+}(M, N)} \m_0$  is Artinian for all  $i\leq t_{R_+}(M, N)$. So, the polynomial $\acute{P}(x)\in \mathbb{Q}[x]$  exists again by \cite{k}. Now, (i) yields that $\deg(\acute{P}(x))\leq i$.
\end{proof}

In the rest of this section we are going to study the asymptotic behavior of $H^i_{\fa}(M, N)$ for $n\to -\infty$, when $i\geq s_{\fa}(M, N)$.

\begin{thm}%-----------------------------------------------------------------------------------------------------------------------------
The $R$-module $R_0/ \fb_0\otimes_{R_0}H^i_{\fa}(M, N)$ is  Artinian for all $i\geq s_{\fa}(M, N)$.
\end{thm}

\begin{proof}
By Lemma \ref{art}, the result is clear for all $i> s_{\fa}(M, N)=:s$. So, it remains to show that  $R_0/ \fb_0\otimes_{R_0}H^s_{\fa}(M, N)$  is Artinian. To do this we use induction on $d:= \dim(N)$.

If  $d= 0$, then $N$ is  $\fa$-torsion. So, in view of \ref{seq}(iii), for all $i\in \N_0$, $R_0/ \fb_0\otimes_{R_0} H^i_{\fa}(M, N)\cong R_0/ \fb_0\otimes_{R_0} \Ext^i_R(M, N)$ is an Artinian $R$-module. Now, let $d> 0$  and assume that the result has been proved for any finitely generated graded $R$-module $N'$  with $\dim(N')= d- 1$. In view of lemma \ref{red}(ii), it suffices to consider the case where  $\Gamma_{\fb_0R}(N)= 0$. Hence, there is an element  $x\in \fb_0$ which is a non-zero divisor on  $N$. Now, consider the exact sequence  $0\to N \stackrel{.x}{\to} N \to N/ xN \to 0$ to get the following exact sequence
$$0\to H^s_{\fa}(M, N)/ xH^s_{\fa}(M, N)\to H^s_{\fa}(M, N/ xN)\to 0:_{H^{s+ 1}_{\fa}(M, N)}x\to 0.$$
Application of the functor  $R_0/ \fb_0\otimes_{R_0} -$ to this sequence induces the exact sequence
$$\Tor_1^R(R_0/ \fb_0, 0:_{H^{s+ 1}_{\fa}(M, N)}x)\to R_0/ \fb_0\otimes_{R_0}H^s_{\fa}(M, N)/ xH^s_{\fa}(M, N)\to R_0/ \fb_0\otimes_{R_0}H^s_{\fa}(M, N/ xN).$$
As a submodule of  $H^{s+ 1}_{\fa}(M, N)$, the module $0:_{H^{s+ 1}_{\fa}(M, N)}x$  is $^*$minimax and $\fa$-torsion. So, the left term of the above sequence is Artinian, by lemma \ref{art}. Also, since $s_{\fa}(M, N/ xN)\leq s$, the right term of this sequence is Artinian, by inductive hypothesis. Thus the middle term $R_0/ \fb_0\otimes_{R_0} H^s_{\fa}(M, N)/ xH^i_{\fa}(M, N)\cong R_0/ \fb_0\otimes_{R_0}H^s_{\fa}(M, N)$ is Artinian, too.
\end{proof}

The following corollary, which generalize \cite[2.8]{DN}, is an immediate consequence of the above theorem.

\begin{cor}%-----------------------------------------------------------------------------------------------------------------------------
For all $i\geq s_{R_+}(M, N)$ the $R$-module $H^i_{R_+}(M, N)$ is tame.
\end{cor}

\begin{lem}%-----------------------------------------------------------------------------------------------------------------------------
Let $\Gamma_{R_+}(N)= 0$  and  $s:= s_{R_+}(M, N)$. If  $\m\notin \Att_{R}(R_0/ \m_0\otimes_{R_0}H^s_{R_+}(M, N)) $, then there is an  $N$-regular element  $x\in R_1$ such that the multiplication map $H^s_{R_+}(M, N)\stackrel{.x}{\to} H^s_{R_+}(M, N)$ is surjective and that $s_{R_+}(M, N/ xN)\leq s- 1$.
\end{lem}

\begin{proof}
Replacing $R_0$ with $R_0[X]_{\m_0[X]}$, we can restrict ourselves to the case where the residue field $R_0/ \m_0$  is infinite. Also, in view of the previous proposition, the set of attached prime ideals of $R_0/ \m_0\otimes_{R_0}H^s_{R_+}(M, N)$  is finite. Set \[\Omega:= (\Att_R(R_0/ \m_0\otimes_{R_0}H^s_{R_+}(M, N))\cup \Ass_R(N))\backslash \Var(R_+).\]
 Then $\Omega$ is a finite set of graded prime ideals of $R$, non of which contains  $R_+$. Therefore, using \cite[1.5.12]{bh}, there exists an element $x\in R_1$  such that  $x\notin \bigcup_{\p\in \Omega} \p$. Now, the long exact sequence
 $$H^i_{R_+}(M, N)(-1)\stackrel{.x}{\to} H^i_{R_+}(M, N)\to H^i_{R_+}(M, N/ xN)\to H^{i+ 1}_{R_+}(M, N)$$
 of graded $R$-modules
implies that $H^i_{R_+}(M, N/ xN)$  is $^*$minimax for all $i> s$. So, it remains to show that $H^s_{R_+}(M, N/ xN)$  is $^*$minimax.
For simplicity set  $H:= H^s_{R_+}(M, N)$. The fact that  $x\notin \bigcup_{\p\in \Att_R(R_0/ \m_0\otimes_{R_0}H)}\p$ implies that  $x H/ \m_0H= H/ \m_0 H$. Therefore, the multiplication map  $H\stackrel{.x}{\to} H$ is surjective and in view of the above exact sequence,  $H^s_{R_+}(M, N/ xN)$ is embedded in the $^*$minimax $R$-module  $H^{s+ 1}_{R_+}(M, N)$, and this completes the proof.
\end{proof}

\begin{thm}%-----------------------------------------------------------------------------------------------------------------------------
Let  $s:= s_{R_+}(M, N)$. Then there is a numerical polynomial $P(x)\in \mathbb{Q}[x]$  of degree less than $s$, such that  $\length_{R_0}(H^{s}_{R_+}(M, N)_n/ \m_0 H^{s}_{R_+}(M, N)_n)= P(n)$ for all  $n\ll 0$.
\end{thm}

\begin{proof}
Since  $H^{s}_{R_+}(M, N)/ \m_0 H^{s}_{R_+}(M, N)$ is Artinian, the numerical polynomial  $P(x)\in \mathbb{Q}[x]$ exists by \cite{k}. It suffices to show that $P(x)$ is of degree less than  $s$. To this end, use the exact sequence
$0\to \Gamma_{R_+}(N)\to N\to N/ \Gamma_{R_+}(N)\to 0$, in conjunction with \ref{seq}(iii),
 to get the long exact sequence
$$\Ext^s_R(M, \Gamma_{R_+}(N))\to H^{s}_{R_+}(M, N)\to H^{s}_{R_+}(M, N/ \Gamma_{R_+}(N))\to \Ext^{s+ 1}_R(M, \Gamma_{R_+}(N)).$$
As  $\Ext^i_R(M, \Gamma_{R_+}(N))$ is finitely generated for all  $i\in\N_0$, it follows that
 $s_{R_+}(M, N/ \Gamma_{R_+}(N))= s$ and that  $H^{s}_{R_+}(M, N)_n\cong H^{s}_{R_+}(M, N/ \Gamma_{R_+}(N))_n$ for all  $n\ll 0$. Therefore, it suffices to consider the case where $\Gamma_{R_+}(N)= 0$.

Let  $H^{s}_{R_+}(M, N)/ \m_0 H^{s}_{R_+}(M, N)= S^1+\cdots +S^k$ be a minimal graded secondary representation with  $\p_j= \sqrt{0:_R S^j}$ for all
$j=1,\cdots, k$. Assume that  $\m= \p_k$. So, $S^k$ is concentrated in finitely many degrees. Hence $P(n)= \length_{R_0}(H^{s}_{R_+}(M, N)_n/ \m_0 H^{s}_{R_+}(M, N)_n)= \length_{R_0}(S^1_n+\cdots +S^{k- 1}_n)$ for all  $n\ll 0$.
This allows us to assume that  $\m\notin \Att_R(\frac{H^{s}_{R_+}(M, N)}{ \m_0 H^{s}_{R_+}(M, N)})$. Therefore, on use of previous lemma, there exists an  $N$-regular element  $x\in R_1$ such that  $s_{R_+}(M, N/ xN)\leq s- 1$. Also, the exact sequence
$$H^{s- 1}_{R_+}(M, N/ xN)_n\to H^{s}_{R_+}(M, N)_{n- 1}\stackrel{.x}{\to} H^{s}_{R_+}(M, N)_n\to 0$$
yields the exact sequence
$$\frac{H^{s- 1}_{R_+}(M, \frac{N}{xN})_n}{\m_0 H^{s- 1}_{R_+}(M, \frac{N}{xN})_n}\to \frac{H^{s}_{R_+}(M, N)_{n- 1}}{ \m_0 H^{s}_{R_+}(M, N)_{n- 1}}\to \frac{H^{s}_{R_+}(M, N)_{n}}{ \m_0 H^{s}_{R_+}(M, N)_{n}}\to 0$$
for all $n\ll 0$.
This allows to conclude by induction on  $s$.

\end{proof}

\end{document}